\documentclass[11pt,a4paper,english]{article}

\usepackage[T1]{fontenc}
\usepackage[utf8]{inputenc}
\usepackage{babel}
\usepackage[final]{microtype} 
\usepackage{lmodern}

\usepackage[a4paper, hmargin={2.7cm,2.7cm},vmargin={3.3cm,3.3cm}]{geometry}
\usepackage{fancyhdr, lastpage}

\usepackage{mathtools} 
\usepackage{amssymb, amsfonts, amsthm, amsxtra, bm}
\usepackage{mathrsfs}  
\usepackage{dsfont}    

\usepackage{graphicx}
\graphicspath{{Figures/}}
\usepackage{tikz}
\usepackage{pgfplots}
\usepgfplotslibrary{fillbetween}
\pgfplotsset{compat=newest}

\definecolor{sienna}{RGB}{160,82,45}
\definecolor{violet}{RGB}{238,130,238}
\definecolor{lightorange}{RGB}{253,208,162}

\usepackage{caption}
\usepackage[margin=0.2cm]{subcaption}
\usepackage{float}

\usepackage{enumitem} 
\setlist{topsep=0.4em, partopsep=0.2em, itemsep=0.1em, parsep=0.05em}
\usepackage{url}
\usepackage{siunitx}
\usepackage{etoolbox} 
\usepackage{cite}

\usepackage[bf,compact,small]{titlesec}
\titlespacing*{\section}{0pt}{14pt}{4pt}
\titlespacing*{\subsection}{0pt}{8pt}{3pt}

\makeatletter
\patchcmd{\ttlh@hang}{\parindent\z@}{\parindent\z@\leavevmode}{}{}
\patchcmd{\ttlh@hang}{\noindent}{}{}{}
\makeatother

\pgfplotsset{every axis/.append style={
    axis x line=middle,
    axis y line=middle,
    axis line style={->},
    xlabel style={at={(ticklabel* cs:1)},anchor=north west},
}}

\DeclarePairedDelimiter{\itvcc}{\lbrack}{\rbrack}      
\DeclarePairedDelimiter{\itvoc}{\lparen}{\rbrack}      
\DeclarePairedDelimiter{\itvco}{\lbrack}{\rparen}      
\DeclarePairedDelimiter{\itvoo}{\lparen}{\rparen}      
\DeclarePairedDelimiter{\abs}{\lvert}{\rvert}          
\DeclarePairedDelimiter{\norm}{\lVert}{\rVert}         
\DeclarePairedDelimiter{\ip}{\langle}{\rangle}         
\DeclarePairedDelimiter{\ceil}{\lceil}{\rceil}         
\DeclarePairedDelimiter{\floor}{\lfloor}{\rfloor}      

\providecommand\given{} 
\newcommand\SetSymbol[1][]{%
   \nonscript\:#1\vert\allowbreak\nonscript\:\mathopen{}}
\DeclarePairedDelimiterX\set[1]\lbrace\rbrace{%
   \renewcommand\given{\SetSymbol[\delimsize]}#1}

\numberwithin{equation}{section}
\allowdisplaybreaks[4]

\newtheorem{theorem}{Theorem}[section]

\newtheorem{lemma}[theorem]{Lemma}
\newtheorem{proposition}[theorem]{Proposition}
\newtheorem{corollary}[theorem]{Corollary}
\theoremstyle{definition}
\newtheorem{definition}{Definition}

\theoremstyle{remark}

\newtheorem{remark}{Remark}

\DeclareMathOperator{\exponential}{e}

\DeclareMathOperator{\im}{im}

\newcommand{\R}{\mathbb{R}}
\newcommand{\C}{\mathbb{C}}
\newcommand{\Z}{\mathbb{Z}}
\newcommand{\N}{\mathbb{N}}

\newcommand{\cF}{\mathcal{F}}

\newcommand{\cG}{\mathcal{G}}

\newcommand{\gaborG}[1]{\mathcal{G}(#1,a,b)} 
\newcommand{\sfrac}[1]{F{(#1)}} 
\newcommand{\round}[1]{R{(#1)}} 
\newcommand{\myexp}[1]{\exponential^{#1}}

\makeatletter
\def\maketimestamp{\count255=\time
\divide\count255 by 60\relax
\edef\thetime{\the\count255:}%
\multiply\count255 by-60\relax
\advance\count255 by\time
\edef\thetime{\thetime\ifnum\count255<10 0\fi\the\count255}
\edef\thedate{\number\day-\ifcase\month\or Jan\or Feb\or Mar\or
             Apr\or May\or Jun\or Jul\or Aug\or Sep\or Oct\or
             Nov\or Dec\fi-\number\year}
\def\timstamp{\hbox to\hsize{\tt\hfil\thedate\hfil\thetime\hfil}}}
\makeatother
\maketimestamp

\fancypagestyle{plain}{%
\fancyhf{} 
\fancyfoot[C]{\small \thepage{} of \pageref{LastPage}}}

\makeatletter
\def\blfootnote{\xdef\@thefnmark{}\@footnotetext} 
\if@titlepage
  \renewenvironment{abstract}{%
      \titlepage \null\vfil
      \@beginparpenalty\@lowpenalty
      \begin{center}\bfseries \abstractname \@endparpenalty\@M \end{center}}%
     {\par\vfil\null\endtitlepage}
\else
  \renewenvironment{abstract}{%
      \small
      \list{}{\setlength{\leftmargin}{3em}\setlength{\rightmargin}{3em}}
      \item[\textbf{\abstractname:}]}
      {\endlist}
\fi
\makeatother

\usepackage{hyperref}
\hypersetup{
    colorlinks=true,
    linkcolor={red!50!black},
    citecolor={blue!50!black},
    urlcolor={blue!80!black},
    pdfview={FitH},
    pdfauthor={Jakob Lemvig},
    pdftitle={A new family of hyperbolic slits in the Gabor frame set of B-spline generators},
}

\usepackage[noabbrev,nameinlink,capitalize]{cleveref}

\begin{document}

\title{A new family of hyperbolic slits in the Gabor frame set of B-spline generators}
\author{Jakob Lemvig\footnote{Technical University of Denmark. E-mail: \protect\url{jakle@dtu.dk}}}
\date{\today}


\maketitle
\thispagestyle{plain}

\begin{abstract}
  We exhibit a new infinite family of hyperbolic curves in the complement of the frame set of Gabor systems with B-spline generators. The proof technique is a combination of an approach by Gr\"ochenig [\emph{Partitions of unity and new obstructions for {G}abor frames}, arXiv:1507.08432, 2015] and a partly partition of unity argument by Nielsen and the author [\emph{Counterexamples to the {B}-spline conjecture for {G}abor frames}, J. Fourier Anal. Appl., 22(6):1440--1451, 2016]. We relate the new hyperbolic obstructions to the ``right bow tie'' of the so-called Janssen tie [\emph{Zak transforms with few zeros and the tie}, In {\em Advances in {G}abor analysis}, Birkh\"auser, 2003].
\end{abstract}


\section{Introduction}
\label{sec:intro}

The Gabor system generated by $g \in L^2(\R)$ with time-frequency shifts along the lattice $a\Z \times b\Z$, $a,b>0$, is defined as
\[
  \gaborG{g} = \set*{\myexp{2\pi i b m \cdot} g(\cdot - a k) \given k,m \in \Z}.
\]
The system $\gaborG{g}$ is called a \emph{Gabor frame} if there exist constants $A,B > 0$ such that
\[
  A \norm{f}^2 \le  \sum_{k,m \in \Z} \abs{\ip{f,\myexp{2\pi i b m \cdot} g(\cdot - a k)}}^2
  \le B \norm{f}^2 \quad \text{for all } f \in L^2(\R).
\]
One of the fundamental problems in Gabor analysis asks, given $g \in L^2(\R)$, for the determination of the \emph{frame set} $\mathcal{F}(g)$, which consists of the parameter values $(a,b)\in\mathbb{R}_+^2$ for which $\mathcal{G}(g,a,b)$ is a frame.
The modern formulation is due to Gr\"ochenig~\cite{MR3232589}, who emphasized it as an important problem in time-frequency analysis; the question itself dates back to the 1990s, see \cite{MR2031050,MR1066587,MR1955931}.

In this paper, we consider the frame set of the cardinal B-splines $N_n$ of order $n \ge 2$, $n \in \N$. The cardinal B-splines are defined as the $n$-fold convolution of the indicator function of the unit interval $\itvcc{0,1}$, i.e.,
\begin{equation*}
  N_1 = \chi_{\itvcc{0, 1}}, \quad \text{and} \quad
  N_{n+1} = N_{n} \ast N_{1}, \quad \text{for } n \in \N.
\end{equation*}
Due to its many desirable properties such as compact support, smoothness, and partition of unity property, these B-splines are widely used as Gabor system generators. The characterization of $\mathcal{F}(N_n)$ was one of six open problems in Gabor analysis posed by Christensen in \cite{ChristensenNew2014}. Our main result, \Cref{thm:new-hyperbolas-frame-set-Bn}, proves the existence of a new infinite family of hyperbolic obstructions to the frame set $\cF(N_n)$.

The geometric complexity of the frame set for the indicator function $N_1$ is well-documented, famously described by the 'Janssen tie' \cite{MR1955931}, and culminating in the complete characterization of $\cF(N_1)$ by Dai and Sun \cite{MR3545108}. Historically, it was unclear whether the complexity of $\cF(N_1)$ was an artifact of the discontinuity of $N_1$.
However, the results in \cite{MR3572909} showed that the frame set of B-spline windows of all orders must have a very complicated structure, sharing several similarities with $\cF(N_1)$. In the present paper, we strengthen this perspective: we argue that certain number-theoretic constraints on $a$ and $b$ determine frame obstructions, and that (at least parts of) the hyperbolic geometry in the Janssen tie is not unique to $N_1$, but rather a general phenomenon for all cardinal B-splines $N_n$.

Our analysis is based on the Zak transform and the Zibulski-Zeevi representation of Gabor systems \cite{MR1448221}. Hence, let us briefly summarize the main idea here. The Zak transform of a function $f \in L^2(\R)$ is defined as
\begin{equation}\label{eq:zakTransform}
  \left(Z_{\lambda}f\right)(x,\gamma)
  = \sqrt{\lambda}\sum_{k\in\mathbb{Z}} f(\lambda(x-
  k))\myexp{2\pi i k \gamma}, \quad a.e.\ x, \gamma \in \mathbb{R},
\end{equation}
with convergence in  $L^2_\mathrm{loc}(\R)$. We only consider rationally oversampled Gabor systems, i.e., $\mathcal{G}(g,a,b)$ with
\[
  ab \in \mathbb{Q}, \quad ab=\frac{p}{q} \quad \gcd(p,q)=1.
\]
For $g\in L^2(\R)$, we define column vectors $\phi^g_\ell(x,\gamma) \in \C^p$ for $\ell
  \in \set{0,1, \dots, q-1}$ by
\[
  \phi^g_\ell(x,\gamma) = \left(p^{-\frac{1}{2}} (Z_{\frac{1}{b}}g)(x-\ell
  \frac{p}{q},\gamma+\frac{k}{p})\right)_{k=0}^{p-1} \ a.e. \ x,\gamma \in \mathbb{R}.
\]
The $p\times q$ matrix defined by $\Phi^g(x,\gamma)=[\phi^g_\ell(x,\gamma)]_{\ell=0}^{q-1}$ is
the so-called Zibulski-Zeevi matrix from which we get the following characterization of the frame property of rationally oversampled Gabor systems:
\begin{theorem}[Zibulski-Zeevi characterization]
  \label{thm:ZZ_singular_values}
  Let $A,B>0$, and let $g \in L^2(\R)$. Suppose $\mathcal{G}(g,a,b)$
  is rationally oversampled Gabor system. Then the following
  assertions are equivalent:
  \begin{enumerate}[label=(\roman*)]
    \item   $\mathcal{G}(g,a,b)$ is a Gabor frame for $L^2(\R)$ with bounds $A$ and $B$,
    \item $\set{\phi^g_\ell(x,\gamma)}_{\ell=0}^{q-1}$ is a frame for $\C^p$ with
          uniform bounds $A$ and $B$ for a.e. $(x,\gamma) \in
            \itvco{0, 1}^2$.
  \end{enumerate}
\end{theorem}

For windows in the Feichtinger algebra $M^1(\R)$, the Zibulski-Zeevi matrix has continuous entries, and thus, to show that $\gaborG{g}$ is not a frame, it suffices to find a single point $(x,\gamma) \in \itvco{0, 1}^2$ such that the set $\set{\phi^g_\ell(x,\gamma)}_{\ell=0}^{q-1}$ does not span $\C^p$.

For $x\in \R$, we let $\round{x}$ denote the round function to the nearest integer, i.e., $R(x)=\floor{x+\frac12}$, and we let $\sfrac{x}=x-\round{x} \in
  \itvoc{-\frac12, \frac12}$ denote the (signed) fractional part of $x$. The main result of \cite{MR3572909} states that $\gaborG{N_n}$ is not a frame for
\begin{equation}
  \label{eq:old_hyperbolic_obstructions}
  ab=\frac{p}{q}<1 \quad \text{for } \abs{\sfrac{b}} \le \frac{1}{nq} \;\text{ and }\; b > \frac{3}{2},
\end{equation}
where $\gcd(p,q)=1$. Based on a large number of computer-assisted symbolic calculations by Kamilla~H.~Nielsen, we realized that there are further hyperbolic obstructions for the frame property of the Gabor system $\mathcal{G}(N_n,a,b)$ ``far'' away from integer values of $b=2,3,4, \dots$. Indeed, we conjectured in \cite{MR3572909} the following: $\mathcal{G}(N_2,a_0,b_0)$ is not a frame for
\begin{equation}
  \label{eq:con_newNonFrame}
  a_0= \frac{1}{2m+1}, \ b_0=\frac{2k+1}{2}, \ k,m\in\mathbb{N}, \ k>m, \ a_0 b_0 < 1,
\end{equation}
and, furthermore, $\mathcal{G}(N_2,a,b)$ is not a frame along the hyperbolas
\begin{equation}
  \label{eq:con_hyperbelStykke}
  \ ab=\frac{2k+1}{2\left(2m+1\right)}, \quad \text{for } \abs{b-b_0} \le \frac{k-m}{2(2m+1)},
\end{equation}
for every $a_0$ and $b_0$ defined by (\ref{eq:con_newNonFrame}). We provided a proof of the case $m=1$ and $k=2$ in \cite{MR3572909}. Gr\"ochenig proved the first part \eqref{eq:con_newNonFrame} of this conjecture in the note \cite{GrochenigPartitions2015}. The full conjecture remained open until Ghosh and Selvan \cite{MR4917072} recently were able to verify it.

Important for the development of the present paper is that Gr\"ochenig~\cite{GrochenigPartitions2015} realized that the pattern \eqref{eq:con_newNonFrame} found in \cite{MR3572909} was part of a much larger family of obstructions that holds for any Gabor system generator in the Feichtinger algebra $M^1(\R)$ satisfying the partition of unity property:

\begin{theorem}[Gr\"ochenig \cite{GrochenigPartitions2015}]
  \label{thm:point-obstructions-grochenig}
  Let $\mu,\nu,r \in \N$ with $r \ge 2$. Let $p=r\nu+j$ and $q=r\mu$, where  $j=1,\dots,r-1$. Define the set $P$ of points $(a_0,b_0) \in \R_+^2$ where $a_0=\frac{r}{q}=\frac{1}{\mu}$ and $b_0=\frac{p}{r} = \nu+\frac{j}{r}$ subject to the constraints that $p$ and $q$ are relatively prime with $q-\mu +1< p<q$.
  If $(a_0,b_0) \in P$ and $g \in M^1(\R)$ satisfies the partition of unity property $\sum_{k \in \Z} g(\cdot + k) = 1$, then $\cG (g,a_0,b_0)$ is \emph{not} a frame for $L^2(\R)$.
\end{theorem}

If one takes $\mu = 2m+1$ odd and $r=2$ in \Cref{thm:point-obstructions-grochenig}, then one recovers the obstructions in \eqref{eq:con_newNonFrame} since $N_2$ satisfies the partition of unity property. Gr\"ochenig noted that $P$ contains countably many point of the form $(1/\mu,\mu-1+j/r)$, $r \in \N$, with accumulation point $(1/\mu, \mu)$ for each $\mu \ge 3$. We refer to \Cref{fig:point-set-P} for a visualization of the set $P$.
\begin{figure}
  \centering
  \includegraphics[width=0.6\textwidth]{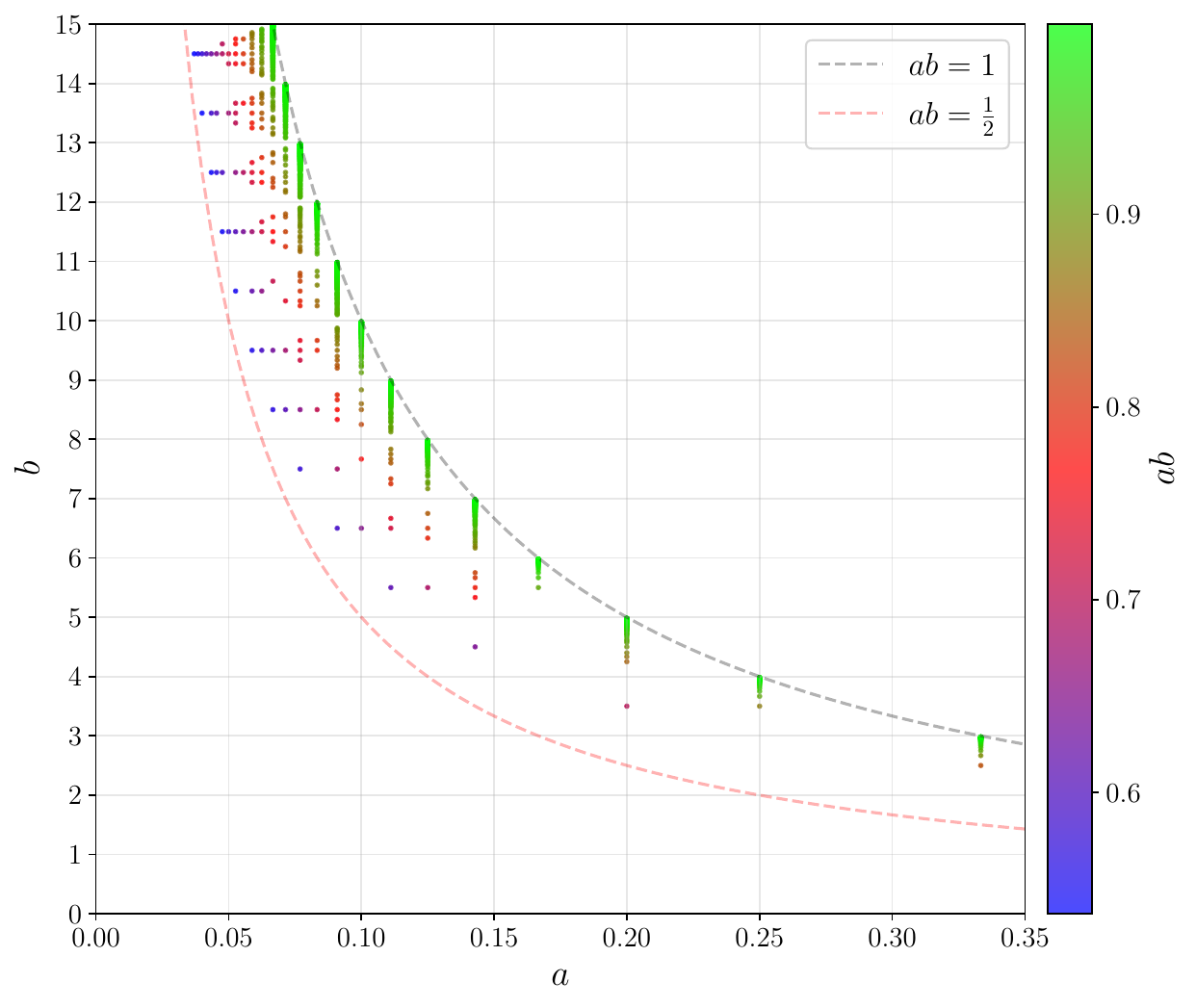}
  \caption{The point obstruction set $P$ defined in \Cref{thm:point-obstructions-grochenig} plotted for $b \le 15$. In each vertical band $a_0 = 1/\mu$, the points become denser and denser as $r$ increases with accumulation point $(1/\mu, \mu)$, $\mu \ge 3$. The colorbar indicates the value of $ab \in \itvoo{1/2,1}$.
  }
  \label{fig:point-set-P}
\end{figure}

The aim of this work is to show that for B-splines the point obstructions defined by $P$ extends to hyperbolic curve obstructions to $\cF(N_n)$ (see \Cref{fig:hyperbola-set}):
\begin{theorem}
  \label{TEST}
  \label{thm:new-hyperbolas-frame-set-Bn}
  Let $n \in\N$. For $(a_0,b_0) \in P$ we set $a_0 b_0=p/q$ and $q=r\mu$ with $p,q,r,\mu$ as in \Cref{thm:point-obstructions-grochenig}. Define the set
  \begin{align}
    H             & = \bigcup_{(a_0,b_0) \in P} H_{(a_0,b_0)}, \label{eq:hyperbolic-obstruction-set}                                      \\
    \intertext{where}
    H_{(a_0,b_0)} & = \set*{(a,b) \in \R_+^2 \given ab=\frac{p}{q} \text{ and } \;\abs{b - b_0} \le \frac{1}{nq}\bigl(\mu-(q-p+1) \bigr)}
    \label{eq:hyperbolic-obstruction-b-interval}
  \end{align}
  If $(a,b) \in H$, then $\cG (N_n,a,b)$ is \emph{not} a frame for $L^2(\R)$. In other words, the complement of the frame set $\R_+^2 \setminus \cF(N_n)$ contains the set $H$.
\end{theorem}

\begin{figure}
  \centering
  \includegraphics[width=0.6\textwidth]{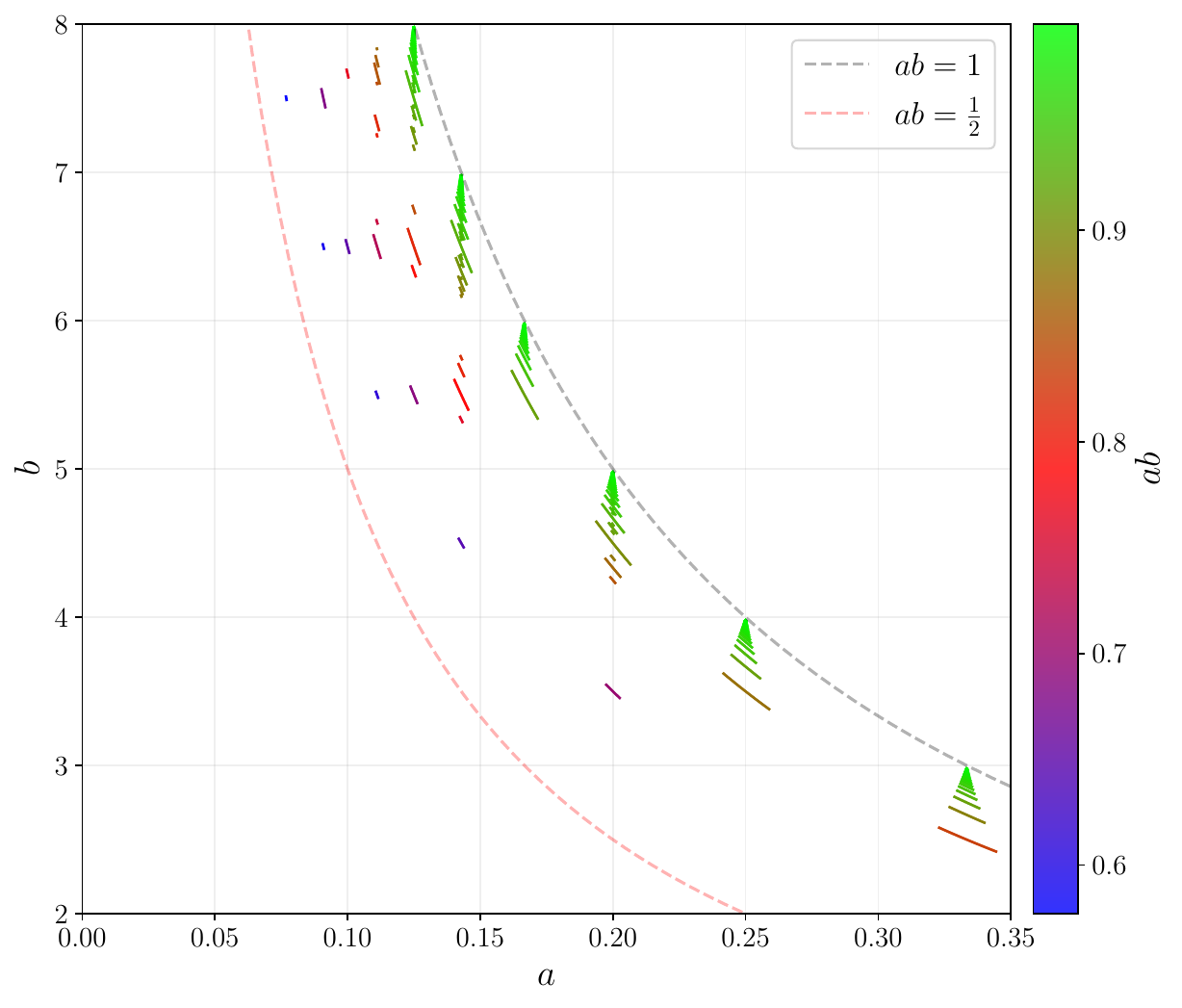}
  \caption{The new hyperbolic obstruction set $H$ for $N_n$, $n=2$, from \Cref{thm:new-hyperbolas-frame-set-Bn}. The hyperbolic segments, defined by \eqref{eq:hyperbolic-obstruction-b-interval}, are colored by their $ab$ values.}
  \label{fig:hyperbola-set}
\end{figure}

In \cite{MR3572909} Nielsen and the author used a zero row failure of the Zibulski-Zeevi matrix to prove frame obstructions, that is, we found a point $(x,\gamma)$ where an entire row of the Zibulski-Zeevi matrix $\Phi^{N_n}$ vanished. Gr\"ochenig's proof in \cite{GrochenigPartitions2015} uses a more powerful idea of explicitly exhibiting $q-p+1$ linearly independent vectors in the kernel of the Zibulski-Zeevi matrix, and this is the proof technique we adopt here.

The proof of \Cref{thm:new-hyperbolas-frame-set-Bn} will be given in \Cref{sec:proof-main-theorem}. First, in the next section, \Cref{sec:point-set-properties}, we will study geometric properties of the sets $P$ and $H$, showcasing the complex behavior of these sets and their connection to the Janssen tie.

For background on Gabor analysis and B-splines we refer to the textbooks \cite{MR3495345,MR1843717}. For results on the frame set of B-splines we refer to
\cite{MR1896414,MR2803840,MR3545108,MR2407006,MR1955931,MR3398948,MR4484792,MR4793698,MR3218799}.

\section{Properties of the obstruction set}
\label{sec:point-set-properties}

It is possible to rewrite the point obstruction set $P$ defined in \Cref{thm:point-obstructions-grochenig} in a simpler and more compact form using a unified parameterization where both $p$ and $q$ represent shifts of the term $r\mu$.

\begin{definition}
  \label{def:point-set-only-mu}
  Let $\mu,k,r \in \N$ be integers satisfying
  \begin{equation}
    \label{eq:conditions-on-mu-r-k}
    \mu \ge 3, \ r \ge 2, \ k \in \set{1,2,\dots,\mu-2}, \ \gcd(k,\mu)=1 \quad \text{and} \quad \gcd(k,r)=1.
  \end{equation}
  Define the set by
  \begin{equation}
    \label{eq:point-set-only-mu}
    P' = \set*{\left(\frac{1}{\mu}, \mu - \frac{k}{r}\right) \in \R_+^2 \given \mu,r,k \text{ satisfying } \eqref{eq:conditions-on-mu-r-k}}.
  \end{equation}
\end{definition}
We associate each point $(a_0,b_0)$ in \eqref{eq:point-set-only-mu} with
\[
  a_0 = \frac{1}{\mu}, \quad b_0 = \mu - \frac{k}{r},
\]
and
\[
  p=r\mu-k, \quad q=r\mu \quad \text{such that} \quad a_0 b_0 = \frac{p}{q}.
\]
Let us argue that the point $(a_0,b_0)$ in \eqref{eq:point-set-only-mu} indeed correspond to the obstructions found by Gr\"ochenig~\cite{GrochenigPartitions2015}:

\begin{lemma}
  The sets $P$ defined in \Cref{thm:point-obstructions-grochenig} and $P'$ defined by \eqref{eq:point-set-only-mu} are identical, i.e., $P=P'$.
\end{lemma}

\begin{proof}
  We first note that $a_0$ is defined identically in both set constructions.

  $P \subset P'$: Assume we have $(p,q)$ as in \Cref{thm:point-obstructions-grochenig}. Set $k = q - p$. From $p<q$ we get $k>0$, hence $k \ge 1$. From $q-\mu+1 < p = q-k$ we get $-\mu+1 < -k$, so $k < \mu-1$, hence $k \le \mu-2$.

  We are given $\gcd(p,q)=1$. Substituting $p=q-k$ gives $\gcd(q-k,q)=1 \iff \gcd(-k,q)=1 \iff \gcd(k,q)=1$. With $q = r\mu$, the relation $\gcd(k, r\mu)=1$ implies $\gcd(k,r)=1$ and $\gcd(k,\mu)=1$.

  $P' \subset P$: Let $p = r\mu - k$ and $q = r\mu$ with $1 \le k \le \mu - 2$ and $\gcd(k,r)=\gcd(k,\mu)=1$. Since $k \ge 1$, we have $p = q - k < q$. Since $k \le \mu - 2$, we get $p = q - k \ge q - (\mu - 2) = q - \mu + 2 > q - \mu + 1$. Thus $q - \mu + 1 < p < q$.
  Relative primality of $p$ and $q$ follows from $\gcd(p,q) = \gcd(r\mu - k, r\mu) = \gcd(k, r\mu) = 1$ because $\gcd(k,r)=\gcd(k,\mu)=1$.

  It remains to show the existence of $\nu,j$ such that $p = r\nu + j$ with $j \in \{1,\dots,r-1\}$. Since $p \not\equiv 0 \pmod r$ (as $\gcd(k,r)=1$), we can write $p = r\nu + j$ with $j = p \bmod r \in \{1,\dots,r-1\}$. Moreover, $p > q - \mu + 1 = r\mu - \mu + 1 \ge 2\mu - \mu + 1 = \mu + 1 > 0$, so $\nu \ge 1$.
\end{proof}

From now on we will denote the point obstruction set simply by $P$, and we will use the parameterization in \Cref{def:point-set-only-mu} in all subsequent proofs and discussions.

As noted in \cite{GrochenigPartitions2015}, the set $P$ accumulates at  $(1/\mu, \mu)$, $\mu \ge 3$, from below. There are no other accumulation points, and all points in $P$ are isolated points. Let us examine some further properties of the point set $P$.
Take $(a_0,b_0) \in P$. Then obviously $a_0 b_0 = \frac{p}{q} < 1$ since $p<q$. Moreover, we have
\[
  a_0 b_0 = \frac{p}{q} = \frac{1}{\mu}\left(\mu - \frac{k}{r}\right) = 1 - \frac{k}{r\mu},
\]
hence also $1/2 < a_0 b_0$ since $1 - \frac{k}{r\mu} \ge 1 - \frac{\mu-2}{2\mu} = \frac{1}{2} + \frac{1}{\mu} > \frac{1}{2}$ for $\mu \ge 3$ and $r \ge 2$.
Thus, all points in $P$ and $H$ lie in the region in  $\R^2_+$ defined by the inequalities $1/2 < a b < 1$.

Obviously, not all rational numbers in the interval $(1/2,1)$ are represented as $a_0 b_0$ as $q$ has to be a composite number of the form $q = r\mu$ with $r \ge 2$ and $\mu \ge 3$, e.g., $q$ cannot be a prime number. However, the allowed inverse densities $a_0 b_0$ are dense since for any $d \in (1/2,1)$ we can take $r=2$, a sufficiently large prime $\mu$ and and odd integer $1 \le k \le \mu -2$ such that $a_0 b_0 = 1 - \frac{k}{r\mu}$ is arbitrarily close to the given value $d$. Moreover, each hyperbola $ab=p/q<1$ has at most finitely many segments $H_{(a_0,b_0)}$ given by \eqref{eq:hyperbolic-obstruction-b-interval} since for fixed $p$ and $q$ there are only finitely many factorizations of $q=r\mu$ with $r \ge 2$ and $\mu \ge 3$. This is in contrast to \eqref{eq:old_hyperbolic_obstructions} where each hyperbola $ab=p/q<1$ contains infinitely many hyperbolic segments of obstructions (one for each integer $b \ge 2$).

The points in the set $P$ are, see \Cref{fig:point-set-P}, distributed in distinct vertical bands $a_0 = \frac{1}{\mu}$ bounded from above by the horizontal line $b = \mu$, where the second coordinate is given by $b_0 = \mu - \frac{k}{r}$ with $k=1,2,\dots,\mu-2$ and $r \ge 2$. From $k \le \mu - 2$, we also get a lower bound on $b_0$ once $\mu$ is fixed:
\begin{equation}
  \label{eq:lower-mu-bound-on-b0}
  \frac{r-1}{r} \mu  < \frac{r-1}{r} \mu + \frac{2}{r} < b_0,
\end{equation}
thus,
\begin{equation}
  \label{eq:mu-bounds-on-b0}
  \tfrac{1}{2}\mu \le \frac{r-1}{r} \mu < b_0 < \mu,
\end{equation}

For $b_0$ confined to a horizontal band $N \le b_0 < N+1$ for some integer $N \ge 2$, we also get bounds on $\mu$. From $N < b_0 < \mu$, we get $\mu \ge N+1$. While, on the other hand, from \eqref{eq:mu-bounds-on-b0} we get $\tfrac{1}{2} \mu < b_0 < N+1$, hence $\mu < 2(N+1)$. Thus, $\mu \in \set{N+1, \dots, 2N+1}$.

With these bounds at hand, we can now analyze ``local gaps'' in $P$ around integer $b$ values. This is a phenomenon complementary to the accumulation of points at $(1/\mu, \mu)$ and shows a distinct asymmetry in the location of the points in $P$.
The proof is based on a version of the Pigeonhole Principle stating that an integer not equal to zero is larger than one as is typical in proofs on Diophantine approximations. The central idea is that from $b_0 - K = (\mu - K) - \frac{k}{r} \neq 0$ for any integer $K$, we get
\begin{equation}
  \abs{b_0 - K} = \frac{\abs{\mu r - K r - k}}{r} \ge \frac{1}{r}
\end{equation}
since $\mu r - K r - k$ is a non-zero integer. Thus, to lower bound the gap $\abs{b_0 - K}$, we need to maximize $r$ subject to the constraints on the parameters. This idea leads to the following local gap bounds:

\begin{proposition}[Local gaps along $a_0=1/\mu$]
  \label{lem:local-gaps-along-a-1-over-mu}
  Let $(a_0,b_0) \in P$ with $a_0 = 1/\mu$. Let $N = \floor{b_0} \ge 2$ be an integer such that $N \le b_0 < N+1$. Then
  \begin{equation}
    \label{eq:local-gap-lower-bound}
    b_0 - N \ge \frac{\mu - N}{\mu - 1}.
  \end{equation}
  and, if $\mu \ge N + 2$, then
  \begin{equation}
    \label{eq:local-gap-upper-bound}
    (N+1) - b_0 \ge \frac{\mu - N - 1}{\mu - 3}.
  \end{equation}
\end{proposition}
\begin{proof}
  We only prove \eqref{eq:local-gap-lower-bound} as the proof of \eqref{eq:local-gap-upper-bound} is similar. Let $\delta = b_0 - N$. Substituting the expression for $b_0 = \mu - \frac{k}{r}$, we get:
  \[
    \delta = (\mu - N) - \frac{k}{r}.
  \]
  Let $M = \mu - N \in \Z$. Since $b_0 < \mu$, it follows that $M$ is a positive integer. We rewrite $\delta$ as:
  \[
    \delta = M - \frac{k}{r} = \frac{Mr - k}{r}.
  \]
  Let $X = Mr - k \in \Z$. Since $\delta > 0$, we have $X > 0$. As $X$ is an integer, we even have that $X \ge 1$.
  Substituting $k = Mr - X$ into the constraint $k \le \mu - 2$ yields:
  \[
    Mr - X \le \mu - 2.
  \]
  and thus
  \begin{equation}
    \label{eq:r-upper-bound}
    r \le \frac{\mu - 2 + X}{M}.
  \end{equation}

  To minimize $\delta = \frac{X}{r}$, we must maximize the denominator $r$. Using the bound derived in \eqref{eq:r-upper-bound}, we get
  \[
    \delta \ge \frac{X}{\frac{\mu - 2 + X}{M}} = \frac{MX}{\mu - 2 + X}.
  \]
  The function $f(x) = \frac{M x}{\mu -2 + x}$ is strictly increasing for $x > 0$. Thus, the minimum gap occurs at the smallest integer $X=1$:
  \[
    \delta \ge f(1) = \frac{M}{\mu - 2 + 1} = \frac{\mu - N}{\mu - 1}.
  \]
\end{proof}

While the points in $P$ accumulate to $(1/\mu, \mu)$ from below along $a=1/\mu$, it follows from \Cref{lem:local-gaps-along-a-1-over-mu} that there are no points in $P$ just above integer $b$ values:

\begin{corollary}[Gaps above integer $b$ values]
  Let $(a_0,b_0) \in P$, and let $N = \floor{b_0} \ge 2$. Then
  \[
    b_0 \not\in \itvco*{N, N + \frac{1}{N}}
  \]
\end{corollary}
\begin{proof}
  Recall that $\mu \in \set{N+1, \dots, 2N+1}$.
  The function $f(x) = \frac{x-N}{x-1}$ is decreasing for $x > N$. Thus, from \eqref{eq:local-gap-lower-bound} we get
  \[
    b_0 - N \ge \frac{\mu - N}{\mu - 1} \ge  f(N+1) = \frac{N+1 - N}{N+1 - 1} = \frac{1}{N}.
  \]
  Hence, $b_0 \not\in \itvco*{N, N + \frac{1}{N}}$.
\end{proof}

\begin{figure}[hb]
  \centering
  \begin{tikzpicture}
    \begin{axis}[
        xlabel=$a$,
        ylabel=$b$,
        width=14.5cm,
        height=6cm,
        grid=major,
        grid style={gray!30},
        axis lines=left,
        xmin=0, xmax=0.34,
        ymin=3.55, ymax=5.45,
        xtick={0, 0.111, 0.2, 0.25},
        xticklabels={$0$, $\frac{1}{2N+1}$, $\frac{1}{N+1}$, $\frac{1}{N}$},
        ytick={3.5, 4, 4.5, 5, 5.5},
        yticklabels={, $N$, $N+\frac{1}{2}$, $N+1$, },
        legend pos=north east,
        legend cell align=left,
        xlabel style={at={(axis description cs:1,-0.05)}, anchor=north},
        ylabel style={at={(axis description cs:-0.05,1)}, anchor=south},
      ]

      \addplot[
        name path=upper,
        domain=0:0.2,
        samples=200,
        color=violet,
        thick,
        dashed,
        dash pattern=on 3pt off 3pt on 1pt off 3pt,
        line width=1.5pt
      ] {4/(1-x)};
      \addlegendentry{$b(1-a)=N$}

      \addplot[
        name path=lower,
        domain=0:0.25,
        samples=200,
        color=sienna,
        thick,
        dashed,
        dash pattern=on 3pt off 3pt on 1pt off 3pt,
        line width=1.5pt
      ] {5/(1+x)};
      \addlegendentry{$b(1+a)=N+1$}

      \addplot[
        domain=0.17:0.3,
        samples=100,
        color=black,
        dashed,
        line width=1.5pt,
        opacity=0.3
      ] {1/x};
      \addlegendentry{$ab=1$}

      \addplot[
        domain=0.085:0.15,
        samples=100,
        color=red,
        dashed,
        line width=1.5pt,
        opacity=0.3
      ] {1/(2*x)};
      \addlegendentry{$ab=1/2$}

      \addplot[
        domain=0:0.25,
        samples=2,
        color=black,
        thick,
        forget plot
      ] {4};

      \addplot[
        domain=0:0.2,
        samples=2,
        color=black,
        thick,
        forget plot
      ] {5};

      \addplot[
        fill=lightorange,
        opacity=0.3,
        forget plot
      ] fill between[of=upper and lower, soft clip={domain=0.111:0.25}];

      \node[font=\large, color=black] at (0.20, 4.5) {$T_N$};
    \end{axis}
  \end{tikzpicture}
  \caption{Sketch of the Janssen tie in $b \in \itvcc{N, N+1}$. The shaded region shows the tile $T_N$.}
  \label{fig:tie-sketch}
\end{figure}
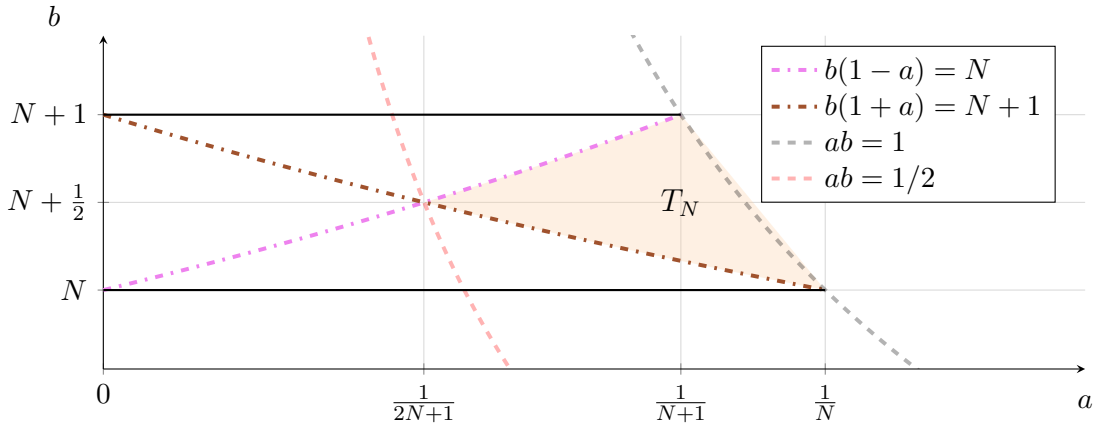
The Janssen tie is a partition of $ab \le 1$ into curvilinear tiles that Janssen \cite{MR1955931} realized were relevant for the study of $\cF(N_1)$. For $ b \in \itvcc{N, N+1}$, $N \in \N$, the tiles are regions bounded by the hyperbolas $b(1-a) = N$, $b(1+a) = N+1$, and $ab=1$ and the lines $a=0$, $b=N$ and $b=N+1$. Of interests in this work is the ``right half'' of the bow tie, namely:
\begin{equation}
  T = \set*{ (a,b) \in \R_+^2 \given \frac{\ceil{b}}{1+a} \le b \le \frac{\floor{b}}{1-a}, \; \; b>2, \; ab < 1 }.
\end{equation}
We denote by $T_N$ each individual tile (see \Cref{fig:tie-sketch}) such that $T=\bigcup_{N=2}^{\infty} T_N$, where
\[
  T_N = \set*{ (a,b) \in \R_+^2 \given \frac{N+1}{1+a} \le b \le \frac{N}{1-a}, \; ab < 1 }
\]
for $N \ge 2$. The $T_N$ region of the tie is where the frame property of the box spline $N_1$ is known to have the most complex behavior~\cite{MR3545108,MR2407006,MR1955931}. Indeed, in this region both rational and irrational $ab$ values can be associated with frame and non-frame properties of $\gaborG{N_1}$ as already reported by Janssen~\cite{MR1955931}.

The geometric structure of the sets $P$ and $H$ can be captured by their containment within these hyperbolic tiles:
\begin{theorem}
  \label{thm:hyperbolic-segments-contained-in-tiles}
  The point set $P$ defined by \eqref{eq:point-set-only-mu} and the hyperbolic obstruction set $H$ defined in \eqref{eq:hyperbolic-obstruction-set} satisfy
  \begin{equation}
    P \subset H \subset T.
  \end{equation}
\end{theorem}
\begin{proof}
  The inclusion $P \subset H$ is obvious from the definitions. On the other hand, the proof of the containment $H \subset T$ will require rather tight bounds, as is illustrated below in \Cref{fig:hyperbola-set_zoom_b_21-23} and, in particular, in \Cref{fig:hyperbola-set_zoom_b_200}. Visually,
  the hyperbolic segments in $H$ appear only barely to be contained within the tiles $T_N$.

  Let $(a,b) \in H$ be a point in the hyperbolic segment with ``center'' $(a_0, b_0) \in P$, $a_0 = 1/\mu$, and $b_0 = \mu - k/r$. Let $N = \floor{b_0}$. We need to verify that $a$ and $b$ satisfy $b(1+a) \ge N+1$ and $b(1-a) \le N$.

  Since $ab = a_0 b_0$ is constant for the entire segment, we can get rid of the variable $a$ and linearize the inequalities in $b$:
  \begin{align*}
    b(1+a) & \ge N+1 \iff b + ab \ge N+1 \iff b \ge (N+1) - a_0 b_0, \\
    b(1-a) & \le N \iff b - ab \le N \iff b \le N + a_0 b_0
  \end{align*}
  Let us first consider the lower boundary $b \ge (N+1) - a_0 b_0$. It suffices to show that the lower endpoint of the segment satisfies this inequality, i.e., that the inequality holds for $b= b_0 - \Delta b$, where $\Delta b = \frac{1}{nr}\left(1 - \frac{k+1}{\mu}\right)$. Hence, we need to verify that
  \begin{equation}
    \label{eq:lower-boundary-check}
    \Delta b \le b_0(1+a_0) - (N+1).
  \end{equation}
  We evaluate the right-hand side of \eqref{eq:lower-boundary-check}:
  \begin{align*}
    b_0(1+a_0) - (N+1) & = \left(\mu - \frac{k}{r}\right)\left(1 + \frac{1}{\mu}\right) - N - 1 \\
                       & = (\mu - N) - \frac{k}{r} - \frac{k}{r\mu}.
  \end{align*}
  Let $X = r(\mu - N) - k \in \Z$. Using $\mu - N - k/r = X/r$, the right-hand side reduces to $\frac{1}{r}\left(X - k/\mu\right)$. Thus, \eqref{eq:lower-boundary-check} can be rewritten as
  \[
    \frac{1}{nr}\left(1 - \frac{k+1}{\mu}\right) \le \frac{1}{r}\left(X - \frac{k}{\mu}\right).
  \]
  which is equivalent to
  \begin{equation}
    \label{eq:lower-boundary-simplified}
    \left(1 - \frac{1}{\mu}\right) + \frac{k}{\mu} (n-1) \le n X. 
  \end{equation}
  From $b_0 = \mu - k/r > N$, we conclude that $X$ is a positive integer, i.e., $X \ge 1$. Thus, since $1 - 1/\mu < 1$ and $\frac{k}{\mu} (n-1) \le n-1$, the inequality \eqref{eq:lower-boundary-simplified} holds which verifies the lower boundary condition.

  To verify the upper boundary condition, we similarly need to show that $\Delta b \le N - b_0(1 - a_0)$.  Let $M = \mu - N \in \Z$. Since $b_0 < N+1$, we have $M \ge 1$. Let $Y = k - r(M-1) \in \Z$. Since $b_0 < N+1$, the integer $Y$ is positive, i.e., $Y \ge 1$. A straightforward calculation shows that the right-hand side reduces to
  \[
    \frac{1}{r}\left(Y - \frac{k}{\mu}\right).
  \]
  From here, the rest of the proof is identical to the lower boundary case, but with $Y$ replacing $X$, thus we omit the details.
\end{proof}

\begin{figure}
  \centering
  \begin{subfigure}[b]{0.48\textwidth}
    \centering
    \includegraphics[width=\textwidth]{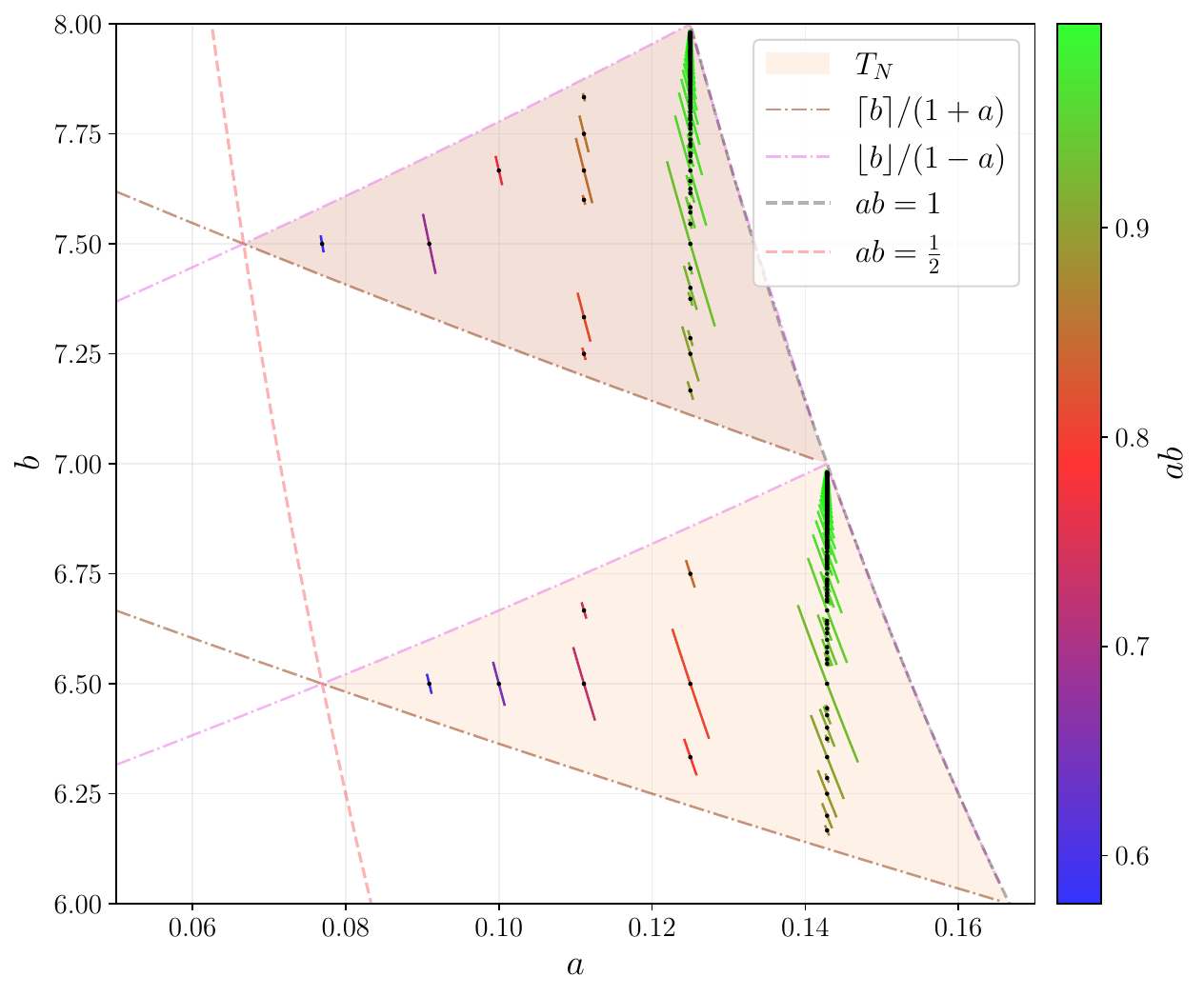}
    \caption{Zoom for $b \in [6,8]$ showing $T_N$ for $N=6$ and $N=7$.}
    \label{fig:hyperbola-set_zoom_b_6-8}
  \end{subfigure}
  \hfill
  \begin{subfigure}[b]{0.48\textwidth}
    \centering
    \includegraphics[width=\textwidth]{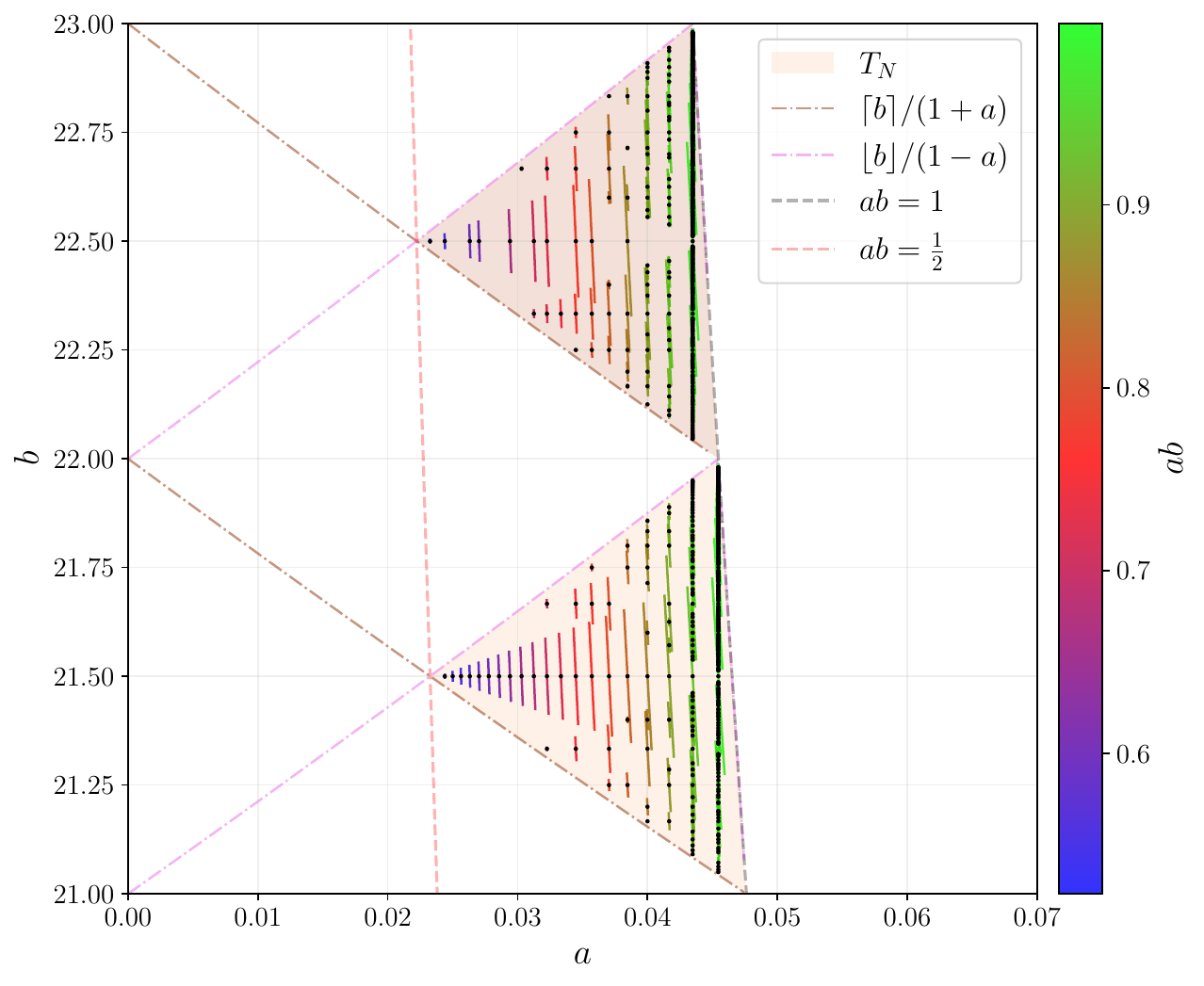}
    \caption{Zoom for $b \in [21,23]$ showing $T_N$ for $N=21$ and $N=22$.}
    \label{fig:hyperbola-set_zoom_b_21-23}
  \end{subfigure}
  \caption{The obstruction point set $P$ (marked by black dots) and the hyperbolic obstruction set $H$ for $N_n$, $n=2$, with hyperbolic segments colored by their $ab$ values. The shaded regions represent the ``right half'' of Janssen's tie $T_N$, whose upper and lower boundary curves are shown in dash-dot lines. Note that the plots are restricted to $r \le 50$ to avoid excessive cluttering at the accumulation points.}
  \label{fig:hyperbola-set_zoom}
\end{figure}

The structure of the point set $P$ and the hyperbolic obstruction set $H$ depend intricately on the
number theoretic properties of $a_0$ and $b_0$ (and thus in turn of $\mu$, $r$, and $k$). As $b$ increases, the structure of $P$ and $H$ increases in complexity.
The reason is that larger $\mu$ values have more divisors, yielding additional valid $r$ values in \Cref{def:point-set-only-mu}, and thus denser obstruction patterns.  This phenomenon is illustrated in the visualizations in
\Cref{fig:hyperbola-set_zoom_b_6-8}, \Cref{fig:hyperbola-set_zoom_b_21-23}, and \Cref{fig:hyperbola-set_zoom_b_200}.
\begin{figure}
  \centering
  \includegraphics[width=0.9\textwidth]{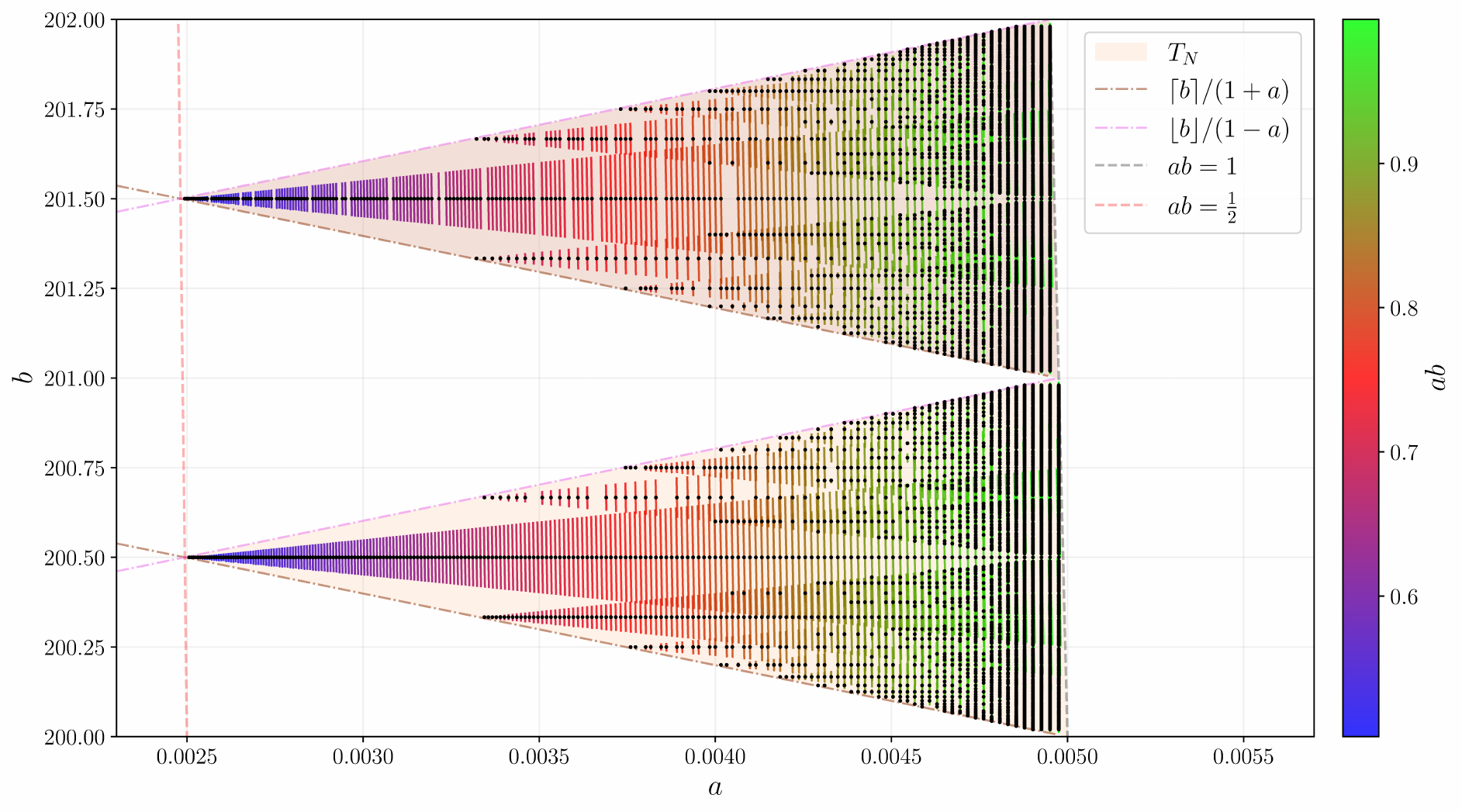}
  \caption{The point set $P$ (black dots) and the hyperbolic obstruction set $H$ for $N_2$ (i.e., $n=2$) for large values of $b$. The hyperbolic segments are colored by their $ab$ values. The complexity of the structure of the sets increases as $b$ increases, compared to \Cref{fig:hyperbola-set_zoom}.}
  \label{fig:hyperbola-set_zoom_b_200}
\end{figure}

The length of each hyperbolic segment $H_{(a_0,b_0)}$ is controlled by
\[
  \abs{b - b_0}  \le \frac{1}{nr} \frac{\mu-k-1}{\mu} = \frac{1}{nq} (\mu-k-1).
\]
Since $k \le \mu - 2$, the possible lengths (in the $b$ direction) of the hyperbolic segments are $\frac{2}{n\mu r} \cdot 1, \frac{2}{n\mu r}\cdot 2, \dots, \frac{2}{n\mu r} (\mu-2)$, where $q=\mu r$. Compared to the hyperbolic obstructions in \cite{MR3572909} of length $\frac{2}{n\mu r}$ (in the $b$ direction), we see that the new obstructions in $H$ can be up to $(\mu-2)$ times longer. Note that, as the order of the B-spline $n$ has a simple scaling effect, we illustrate only the case $n=2$ in the figures.

Let us finally return to the accumulation points of $P$, and see how the hyperbolic obstruction curves in $H$ behave near these points. In \Cref{fig:hyperbolas-r-colord-black-centers}, the hyperbolic obstruction set $H$ is visualized for $\mu = 5$ near the accumulation point $(1/5,5)$. The hyperbolic segments accumulate to the point $(1/5,5)$ from below in the band $a_0 = 1/5$ with arc lengths smaller than $2/(nr)$. As $r$ increases during this accumulation, the segment lengths decrease, approaching zero in the limit.

\begin{figure}
  \centering
  \includegraphics[width=0.6\textwidth]{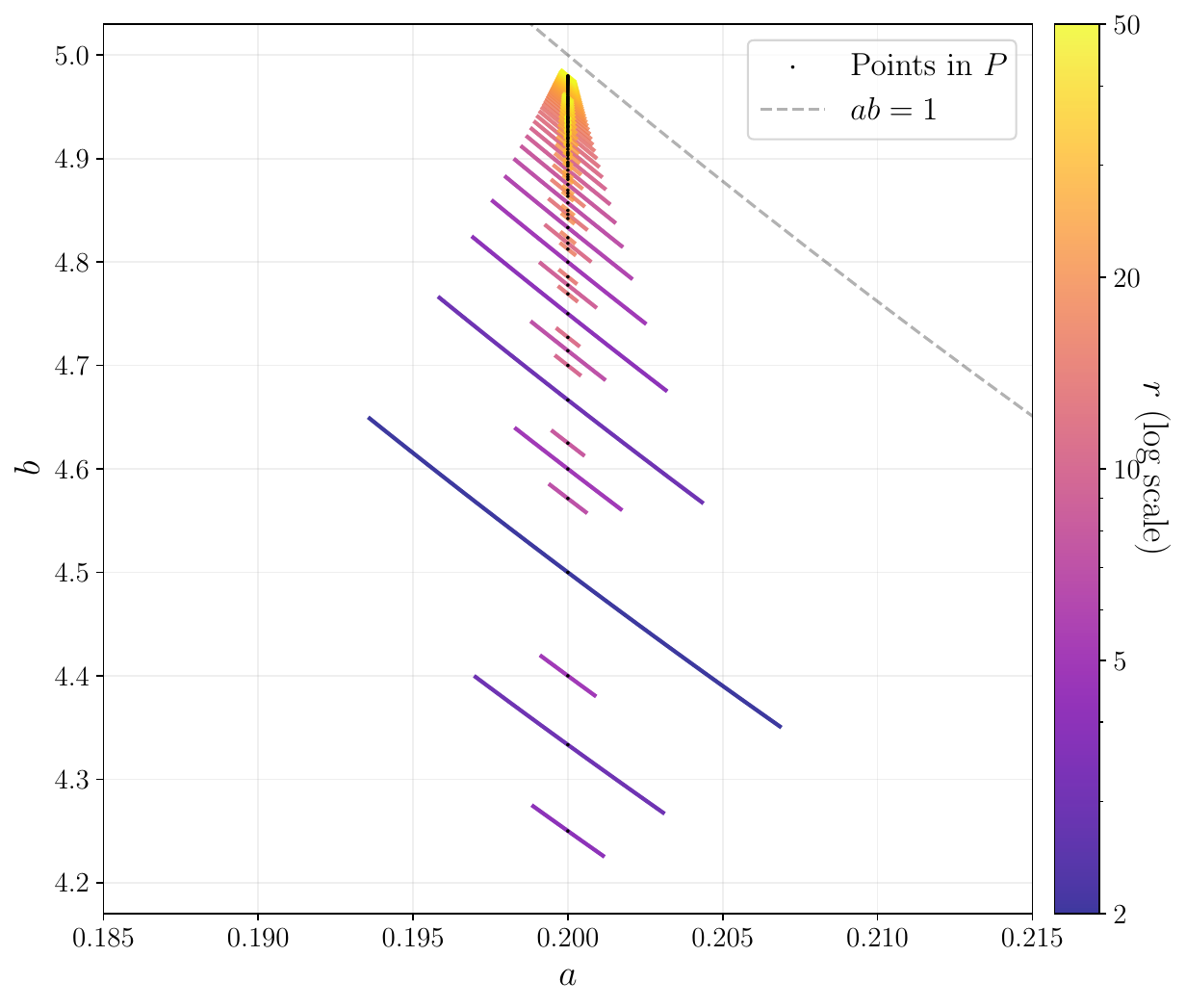}
  \caption{Zoomed view of the hyperbolic obstruction set $H$ for $N_n$, $n=2$, near the accumulation point $(1/5, 5)$ ($\mu = 5$). The points $P$ are marked with black dots, and the hyperbolic curves are colored by their $r$ values (log scale). Larger $r$ values produce shorter obstruction curves in $H$ as the length of each hyperbolic segment is bounded by $2/(nr)$. The plot is restricted to $r \le 50$ for visibility hence the accumulation is truncated.}
  \label{fig:hyperbolas-r-colord-black-centers}
\end{figure}

\section{Proof of Theorem~\ref{TEST}}
\label{sec:proof-main-theorem}

We need the following result from \cite{MR3572909} that gives a partly partition of unity property for B-splines under dilations when the dilation parameter $c$ is close to an integer.

\begin{lemma}[\!\cite{MR3572909}]
  \label{lem:partly-pou}
  Let $n \in \N$ and $c>0$. Assume that $\abs{\sfrac{c}} \le
    \frac{1}{n}$.
  \begin{enumerate}[label=(\roman*)]
    \item If $\sfrac{c}\ge 0$, then
          \begin{equation}
            \label{eq:partly-part-of-unity-1}
            \sum_{k \in \Z} N_n((x+k)/c) = \mathrm{const} \quad \text{for }
            x \in \bigcup_{m\in \Z}\itvcc{m+n\sfrac{c}, m+1}
          \end{equation}
    \item If $\sfrac{c}\le 0$, then
          \begin{equation}
            \label{eq:partly-part-of-unity-2}
            \sum_{k \in \Z} N_n((x+k)/c) = \mathrm{const} \quad \text{for }
            x \in \bigcup_{m\in \Z}\itvcc{m, m+1+n\sfrac{c}}
          \end{equation}
  \end{enumerate}
\end{lemma}

Let $e_m$ denote the column
vector with $e_m(m)=1$ and $e_m(\ell)=0$ for $\ell\neq m$. 
We first need a cancellation property for the Zibulski-Zeevi matrix $\Phi^{N_n}(x,\gamma)$ of
B-splines. The entries of the matrix $\Phi^{N_n}(x,\gamma)$
are expressed in terms for $Z_{1/b}N_n$. However,
the sum of every $\mu$th column can be
expressed in terms of $Z_{1/(rb)}N_n$. This is useful since $rb$ is
assumed to be close to an integer and since, by Lemma~\ref{lem:partly-pou}, we have
good control of the zeros of the Zak
transform $Z_{1/c} N_n$ for any $c$ close to an
integer. From these ideas, we can prove the following cancellation lemma.

\begin{lemma}
  \label{lem:ZZ-structure}
  Let $n \in\N$. Let $\mu,r, k \in \N$ be as in \Cref{def:point-set-only-mu}. Set $p=r\mu-k$, $q=r\mu$, and $b_0 = \mu - k/r$.
  Let $b>3/2$ be given such that $\abs{b-b_0} \le \frac{1}{nr}$, i.e., $\abs{\sfrac{rb}} \le \frac1n$.
  \begin{enumerate}[label=(\roman*)]
    \item If $\sfrac{rb}\ge 0$, there exists a
          constant $K$ such that 
          \[
            \sum_{\ell=0}^{r-1} \phi_{\ell \mu}^{N_n}(\frac{x}{r},0) = K e_0 \in \C^p
            \qquad \text{for $x \in \bigcup_{m \in
                  \Z}\itvcc{m+n\sfrac{rb}, m+1}$.}
          \]
          \label{item:ZZ-cancellation-pos}
    \item If $\sfrac{rb}\le 0$, there exists a
          constant $K$ such that 
          \[
            \sum_{\ell=0}^{r-1} \phi_{\ell \mu}^{N_n}(\frac{x}{r},0) = K e_0 \in \C^p
            \qquad \text{for } x \in \bigcup_{m \in
              \Z}\itvcc{m, m+1+n\sfrac{rb}}
          \]
          \label{item:ZZ-cancellation-neg}
  \end{enumerate}
\end{lemma}
\begin{proof}
  We will only prove \ref{item:ZZ-cancellation-pos} as the proof of \ref{item:ZZ-cancellation-neg} is similar. Note that $\round{rb}=p$. Hence, $rb$ is ``close'' to the integer $p$ by $0 \le \sfrac{rb} \le \frac1n$.

  For each $k=0,1, \dots, p-1$ we compute:
  \begin{align*}
    \sum_{\ell=0}^{r-1} \Phi^{N_n}(x,0)_{k,\ell \mu} & =
    \sum_{\ell=0}^{r-1} p^{-1/2} Z_{\tfrac{1}{b}}N_n(x-\ell \mu
    \frac{p}{q},\frac{k}{p})                                                             \\ &= (bp)^{-1/2} \sum_{\ell=0}^{r-1}
    \sum_{z \in \Z} N_n(\tfrac{1}{b}(x-
    \frac{p}{r}\ell-z)) \myexp{2\pi i z \frac{k}{p}}
    \intertext{Write $z=pv+w$ for $v\in \Z$ and $w=0,1,\dots, p-1$, and
      continue:}
                                                     & = (bp)^{-1/2} \sum_{\ell=0}^{r-1}
    \sum_{v \in \Z} \sum_{w=0}^{p-1} N_n(\tfrac{1}{rb}(rx-p\ell
    -rpv-rw)) \myexp{2\pi i w \frac{k}{p}}
    \intertext{Write $z=rv+\ell$ for $v\in \Z$ and $\ell=0,1,\dots, r-1$, and
      continue:}
                                                     & = (bp)^{-1/2} \sum_{w=0}^{p-1}
    \sum_{z \in \Z}  N_n(\tfrac{1}{rb}(r(x-w)-pz)) \myexp{2\pi i w \frac{k}{p}}
  \end{align*}
  We claim that there exists a constant $d \in \R$ such that for every
  $w\in \set{0,1,\dots,p-1}$:
  \[
    \sum_{z \in \Z}  N_n(\tfrac{1}{rb}(r(x-w)-pz)) = d \qquad \text{for }  rx \in
    \bigcup_{m \in \Z} \itvcc{m+n\sfrac{rb}, m+1}.
  \]
  Fix $w \in \set{0,1,\dots,p-1}$ for a moment. Recall that $p = \round{rb}$.
  Since
  \[
    \frac{1}{rb}(r(x-w)-pz)=\frac{\round{rb}}{rb}\biggl(\frac{r(x-w)}{\round{rb}}+z
    \biggr),
  \]
  it follows from Lemma~\ref{lem:partly-pou} that $\sum_{z \in \Z}  N_n(\tfrac{1}{rb}(r(x-w)-pz))$ is constant for
  \[
    \frac{r(x-w)}{\round{rb}} \in \bigcup_{m\in \Z}\itvcc{m+n\sfrac{\tfrac{rb}{\round{rb}}}, m+1},
  \]
  that is, for
  \[
    \frac{r}{p} x \in \left(\bigcup_{m\in \Z}\itvcc{m+n\frac{\sfrac{rb}}{p}, m+1}\right)+\frac{r}{p}w.
  \]
  Hence, since $n\frac{\sfrac{rb}}{p}\le \frac{1}{p}$ and since $r$ and $p$ are relatively prime, we see that $\sum_{z \in \Z}  N_n(\tfrac{1}{rb}(r(x-w)-pz))$ is constant
  for
  \[
    \frac{r}{p} x \in \bigcap_{w=0}^{p-1}\left(\bigcup_{m\in
      \Z}\itvcc{m+n\frac{\sfrac{rb}}{p}, m+1}\right)+\frac{r}{p}w =
    \bigcup_{m\in \Z}\itvcc{\frac{m+n\sfrac{rb}}{p}, \frac{m+1}{p}}.
  \]
  This completes the proof of the claim. 

  Hence, for $x \in \bigcup_{m \in
      \Z}\itvcc{m+n\sfrac{rb}, m+1}$, we can
  conclude:
  \begin{equation*}
    \sum_{\ell=0}^{r-1} \Phi^{N_n}(x/r,0)_{k,\ell \mu}
    =  (bp)^{-1/2} \sum_{w=0}^{p-1} d \myexp{2\pi
      i w \tfrac{k}{p}} = \begin{cases}
      b^{-1/2} d p^{1/2} & k=0,      \\
      0                  & k \neq 0.
    \end{cases}
  \end{equation*}
\end{proof}

With the cancellation property in Lemma~\ref{lem:partly-pou} at hand,
we can follow \cite{GrochenigPartitions2015} to complete the proof of \Cref{thm:new-hyperbolas-frame-set-Bn}.

\begin{proof}[Proof of Theorem~\ref{thm:new-hyperbolas-frame-set-Bn}]
  Assume that $\sfrac{rb} \ge 0$. Recall that $q-p+1\le \mu-1$.
  The ``samples'' of
  $\Phi^{N_n}(x,0)$ in the $x$ variable in the columns indexed by
  $s\in \{0,1,\dots,q-p+1\}$ are within an interval of length
  $\frac{q-p+1}{q}$ (up to quasi-periodicity).
  Since,  by assumption,
  \[
    \frac{q-p+1}{\mu}\le (1-n\sfrac{rb})
  \]
  or equivalently
  \[
    \frac{q-p+1}{q}\le \frac{1}{r} (1-n\sfrac{rb}),
  \]
  we see from \Cref{lem:ZZ-structure} that for  $s\in \{0,1,\dots,q-p+1\}$
  \[
    \sum_{\ell=0}^{r-1} \Phi^{N_n}(x_0,0)_{k,\ell \mu+s} =
    \begin{cases}
      K_s & k=0  ,    \\
      0   & k \neq 0,
    \end{cases}
  \]
  for some $x_0$. Define $v_s=(1/K_s)\cdot \sum_{\ell=0}^{r-1}e_{l\mu+s}$ for
  $s=0,1,\dots,\mu-1$. Then, $\Phi^{N_n}(x_0,0)v_s=\delta_0 \in \C^p$ for  $s\in
    \{0,1,\dots,q-p+1\}$.

  We have found $q-p+1$ linear independent vectors
  $v_s-v_0$, $s=1,2,\dots,q-p+1$, in the kernel of $\Phi^{N_n}(x_0,0)$.
  From the rank-nullity theorem it then follows that:
  \[
    \dim{(\im{(\Phi^{N_n}(x_0,0))})}= q-\dim{(\ker{(\Phi^{N_n}(x_0,0))})}
    \le q-(q-p+1)=p-1.
  \]
  Hence, $\set{\phi^{N_n}_\ell(x_0,0)}_{\ell=0}^{q-1}$ is not a spanning set
  for $\C^p$, in particular, it is not a frame for $\C^p$.
  By Theorem~\ref{thm:ZZ_singular_values}, we conclude that $\gaborG{N_n}$
  is not a frame for $L^2(\R)$. The proof of the case $\sfrac{rb} < 0$ is similar.
\end{proof}

\begin{remark}
  In the proof of \Cref{thm:new-hyperbolas-frame-set-Bn}, we tacitly used that the Zak transform $Z_\lambda N_n$ is continuous. This is not true for $n=1$. However, $Z_\lambda N_1$ is continuous with respect to the second variable and piecewise continuous with respect to the first variable with finitely many jump discontinuities. It follows that we can still conclude that the minimum of the smallest singular value of the Zibulski-Zeevi matrix $\Phi^{N_1}(x,\gamma)$ over $(x,\gamma)\in \itvco{0,1}^2$ is zero, which is the conclusion we need.
\end{remark}
\begin{remark}
  The obstructions in \Cref{thm:new-hyperbolas-frame-set-Bn} have been known to the author for a number of years. They were originally presented at \emph{The International Workshop on Operator Theory and Applications (IWOTA)} in 2016, but have not previously appeared in print.
\end{remark}

Let us return to the conjecture in \eqref{eq:con_hyperbelStykke}, now a theorem in \cite{MR4917072} on frame obstructions for the hat spline $N_2$. To compare \eqref{eq:con_hyperbelStykke} with our result, we must restrict \Cref{thm:new-hyperbolas-frame-set-Bn} to the special setting of $n=2$, $r=2$, and odd $\mu$. For $ab=p/q$, the bound in \eqref{eq:con_hyperbelStykke} can be written as $\abs{b - b_0} \le \frac{1}{nr} \left(\mu  - k\right)/\mu$, whereas \Cref{thm:new-hyperbolas-frame-set-Bn} gives $\abs{b - b_0} \le \frac{1}{nr} \left(\mu  - k - 1\right)/\mu$ (with $nr=4$). The segments from \eqref{eq:con_hyperbelStykke} are thus slightly longer by $1/(4\mu)$ in each $b$-direction. This indicates that the obstructions in \Cref{thm:new-hyperbolas-frame-set-Bn} are suboptimal, which is not in contradiction with the method: the proof does not \emph{determine} the kernel of the Zibulski-Zeevi matrix, it only exhibits $q-p+1$ linearly independent vectors in the kernel.

We end the paper with a general remark on the nature of the known general obstructions for B-splines and other functions that generate partitions of unity. For functions $g \in M^1(\R)$ that generate partitions of unity, there are two types of obstructions to the frame property:
\begin{enumerate}[label=(\Roman*)]
  \item For $b=2,3,\dots$ and any $a>0$ the Gabor system $\gaborG{g}$ is not a frame \cite{MR2657413,MR2045812,MR1752589}  \label{item:integer-obstruction},
  \item For $(a,b)$ in the point set $P$ the Gabor system $\gaborG{g}$ is not a frame \cite{GrochenigPartitions2015} \label{item:point-obstruction}.
\end{enumerate}
When viewed along hyperbolas $ab=p/q<1$, these two types of obstructions are both point obstructions. For B-splines, the hyperbolic obstructions in \cite{MR4917072} are ``grown'' out of the first type of obstruction \ref{item:integer-obstruction} along $ab=p/q<1$, whereas the obstructions in \Cref{thm:new-hyperbolas-frame-set-Bn} are ``grown'' out of the second type of obstruction \ref{item:point-obstruction}, also along $ab=p/q<1$. In the opposite direction, as the B-spline order $n$ increases, these hyperbolic obstructions contract, degenerating to the point obstructions 
as $n \to \infty$. To see why B-splines admit these hyperbolic extensions of point obstructions, note first that any function $g$ that generates a partition of unity $\sum_{k \in \Z} g(\cdot + k) = 1$ will, for integer $c$, also satisfy $\sum_{k \in \Z} g((\cdot+k)/c) = c$. B-splines further satisfy a partly partition of unity property in 
Lemma~\ref{lem:partly-pou} for $c$ close to an integer -- that is, $\sum_{k \in \Z} N_n((\cdot+k)/c)$ is constant on some subinterval of $\itvco{0,1}$ -- which is the key additional property that allows  the point obstructions to be extended into hyperbolic segments. In fact, any sufficiently nice function $g$ that generates a partition of unity and retains such a partly partition of unity under dilations $(1/c)\Z$ for $c \approx 1$ will have similar hyperbolic obstructions.

\section*{Declaration of AI use}
During the preparation of this work the author used Gemini (Google) and GitHub Copilot in order to improve the readability and grammatical accuracy of the manuscript, proofread and check mathematical arguments, and assist with writing the code to produce the figures. After using these tools/services, the author reviewed and edited the content as needed and takes full responsibility for the content of the published article.

\def\cprime{$'$} \def\cprime{$'$} \def\cprime{$'$} \def\cprime{$'$}
  \def\uarc#1{\ifmmode{\lineskiplimit=0pt\oalign{$#1$\crcr
  \hidewidth\setbox0=\hbox{\lower1ex\hbox{{\rm\char"15}}}\dp0=0pt
  \box0\hidewidth}}\else{\lineskiplimit=0pt\oalign{#1\crcr
  \hidewidth\setbox0=\hbox{\lower1ex\hbox{{\rm\char"15}}}\dp0=0pt
  \box0\hidewidth}}\relax\fi} \def\cprime{$'$} \def\cprime{$'$}
  \def\cprime{$'$} \def\cprime{$'$} \def\cprime{$'$} \def\cprime{$'$}


\end{document}